\documentclass{article}

\usepackage{graphicx}
\usepackage{amssymb}
\usepackage{amsmath}
\usepackage{amsthm}
\usepackage{float}
\usepackage[skip=2pt,font=small]{caption}
\usepackage{sagathm}
\usepackage{authblk}

\newcommand{\R}{\mathbb{R}}

\newcommand{\sst}{\subseteq}
\newcommand{\eps}{\varepsilon}

\newcommand{\dimH}{\dim_\mathrm{H}}

\newcommand{\tp}{\text{.}}
\newcommand{\tc}{\text{,}}

\makeatletter
\def\thanks#1{\protected@xdef\@thanks{\@thanks
        \protect\footnotetext{#1}}}
\makeatother

\begin{document}

\title{A Gasket Construction of the Koch Snowflake and Variations}

\author[]{Robert C.\ Sargent\thanks{Email: rsargent@umd.edu}}
\affil[]{University of Maryland, College Park}

\date{February 2, 2025}

\maketitle

\vspace{-14pt}
\begin{abstract}
    We introduce a construction of the Koch snowflake that is not inherently six-way symmetrical, based on iteratively placing similar rhombi. This construction naturally splits the snowflake into four identical self-similar curves, in contrast to the typical decomposition into three Koch curves. Varying the shape of the rhombi creates a continuous family of new fractal curves with rectangular symmetry. We compute the Hausdorff dimension of the generalized curve and show that it attains a maximum at the original Koch snowflake.
\end{abstract}

\vspace{6pt}

\vspace{6pt}

\vspace{6pt}
\noindent
Thanks to Dr.\ Lawrence Washington (University of Maryland) for mentorship, encouragement, and feedback. Thanks to Dr.\ Yann Demichel (Universit\'e Paris Nanterre) for feedback and advice.

\section{Introduction}\label{sec:intro}
The standard construction of the Koch snowflake defines the fractal to be the limit of a sequence of curves. This paper presents a new construction that is similar to those of the Sierpi\'nski gasket and Apollonian gasket. We refer to this construction as a ``gasket construction'' because it starts with a region that is viewed as ``empty space'' into which polygons are inserted, leaving smaller empty spaces, which are filled by more polygons, and so on. To construct the Koch snowflake, the initial region is a horizontally oriented rhombus, and the polygons that are inserted are vertical and horizontal rhombi, with the orientation alternating at each iteration.
\begin{figure}[H]
    \centering
    \includegraphics[width=0.85\linewidth]{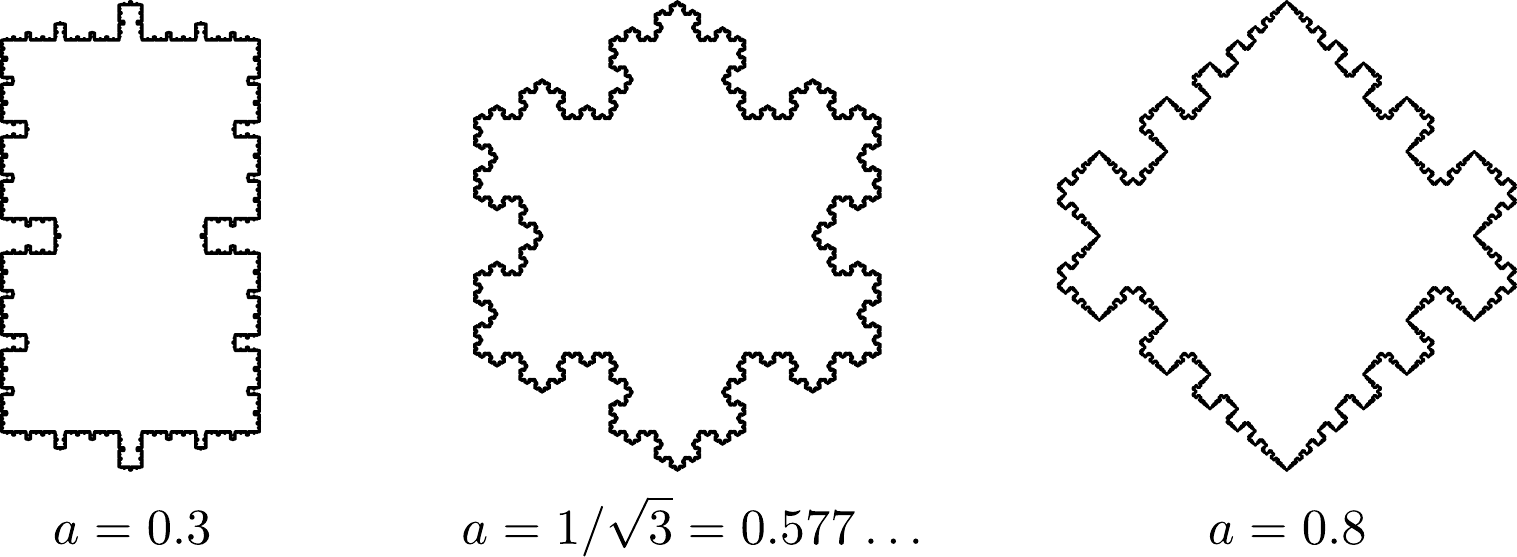}
    \caption{The fractal created by the gasket construction at different values of $a$, where $a$ is the aspect ratio of the rhombi in the construction.}
    \label{fig:fig1}
\end{figure}

The standard Koch snowflake construction starts with an equilateral triangle (\cref{fig:kochorig}). Triangles of one-third the size are placed on the center of each edge of the original triangle, creating a twelve-sided polygon. Triangles one-third the size of those are placed on each edge of this polygon, and so on. The Koch snowflake is the limit of this sequence of polygonal curves. The Koch snowflake is contrasted with the Koch \emph{curve}, which is constructed in the same manner as the snowflake, but starts from a line segment instead of a triangle. Thus, the Koch snowflake is the union of three identical Koch curves, each starting from one side of the initial triangle.
\begin{figure}[H]
    \centering
    \includegraphics[width=0.9\linewidth]{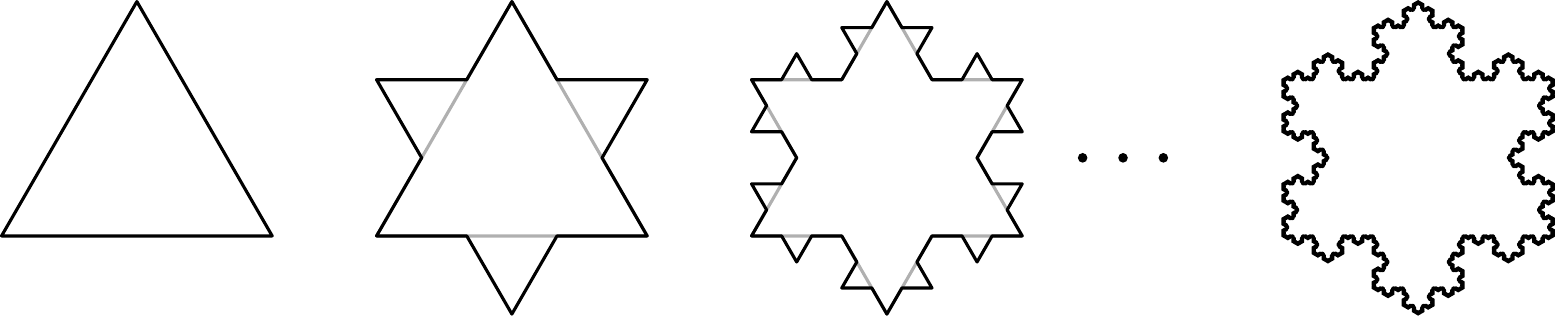}
    \caption{The standard construction of the Koch snowflake.}
    \label{fig:kochorig}
\end{figure}

Many generalizations of the Koch curve exist, including a family described by von Koch himself \cite{koch1906}, two years after he introduced the curve in 1904 \cite{ogkoch}. One generalization varies the number of sides of the polygons added and the size of the polygons added \cite{nc2,nc}. Some vary the number of polygons added at each step \cite{rand,joel}. Others vary where along each segment the new geometry is added \cite{koch1906,sys}, as well as the angles in the segments added \cite{var,kochawave}. Some generalizations take the snowflake into the third dimension \cite{3d,nature,joel}.

What these examples have in common is that they are based on the standard construction of the Koch snowflake. In particular, they reflect the standard decomposition of the snowflake into three Koch curves, either directly generalizing the Koch curve or supporting an analogous decomposition. This causes the resulting fractals to have symmetry analogous to the rotational symmetry of the Koch snowflake. By contrast, the gasket construction introduced here generalizes the \emph{rectangular} symmetry of the snowflake. The construction naturally gives a decomposition of the Koch snowflake into four congruent self-similar curves, as opposed to the typical three (\cref{fourdecomp}). This method creates a family of curves that have rectangular symmetry but not three-fold or six-fold rotational symmetry. The original Koch snowflake, which does have six-fold symmetry, is a member of this family, but this symmetry only appears in the limit of the construction.

\begin{figure}[H]
    \centering
    \includegraphics[width=0.7\linewidth]{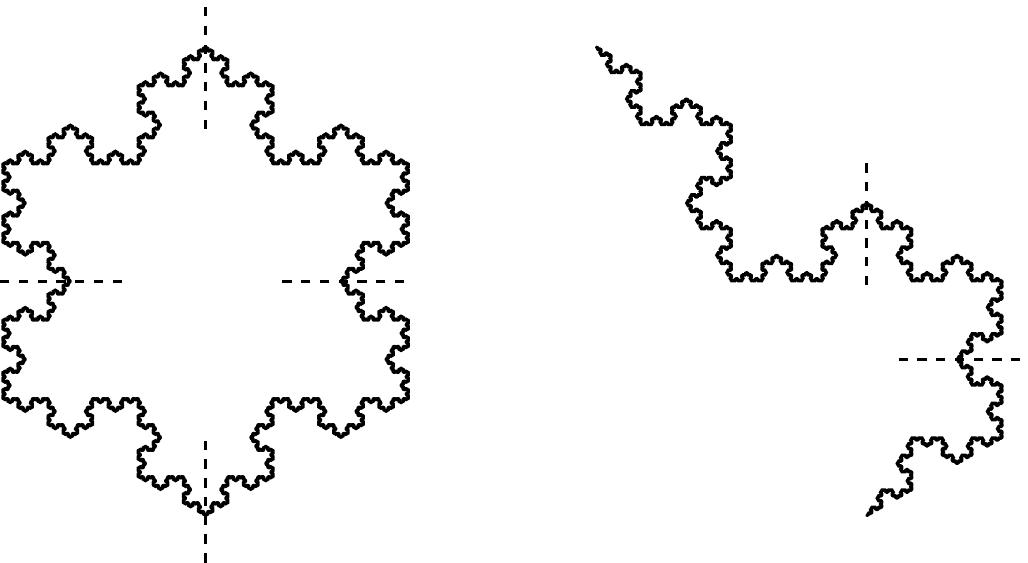}
    \caption{The Koch snowflake as the union of four congruent self-similar curves. Each of the four curves decomposes into three parts similar to the whole.}
    \label{fourdecomp}
\end{figure}

\Cref{sec:gasket} rigorously defines the fractal as the intersection of a sequence of sets, then proves that the construction yields a Jordan curve that depends continuously on the aspect ratio of the rhombi. \Cref{sec:koch} proves that the gasket construction does result in the same curve as the standard Koch snowflake construction. \Cref{sec:hausdorffdim} computes the Hausdorff dimension of the fractal as a function of the aspect ratio, using standard fractal-geometric tools. This includes the surprising result that the Hausdorff dimension is maximized when the aspect ratio is $\sqrt{3}:1$, exactly when the construction yields the standard Koch snowflake.

\section{The gasket construction}\label{sec:gasket}
The gasket construction starts with a horizontally oriented rhombus of aspect ratio $1:a$, with $a \in (0,1)$ fixed. This rhombus has height 2 and width $2/a$, and is centered at the origin. The interior of this rhombus is viewed as ``empty space'' into which rhombi are inserted at each iteration. All rhombi in the construction are either horizontal ($1:a$) or vertical ($a:1$). Next, the largest possible vertical rhombus is placed within the first rhombus, splitting the empty space into two polygons; this completes iteration 1. The polygons that make up this space are called \emph{empty polygons}. (We call these polygons ``empty'' because they are the ones we fill with rhombi in future iterations. In figures, we typically represent these empty polygons as shaded regions.) At iteration 2, each of the two empty polygons have the largest possible horizontal rhombi placed into them, leaving four empty polygons. In general, at iteration $k$, a rhombus is placed into each of the $2^{k-1}$ empty polygons, leaving $2^k$ empty polygons, with the orientation of the rhombi alternating with the parity of $k$. (For convenience, we do not consider the initial horizontal rhombus to be an empty polygon.)
\begin{figure}[H]
    \centering
    \includegraphics[width=0.95\linewidth]{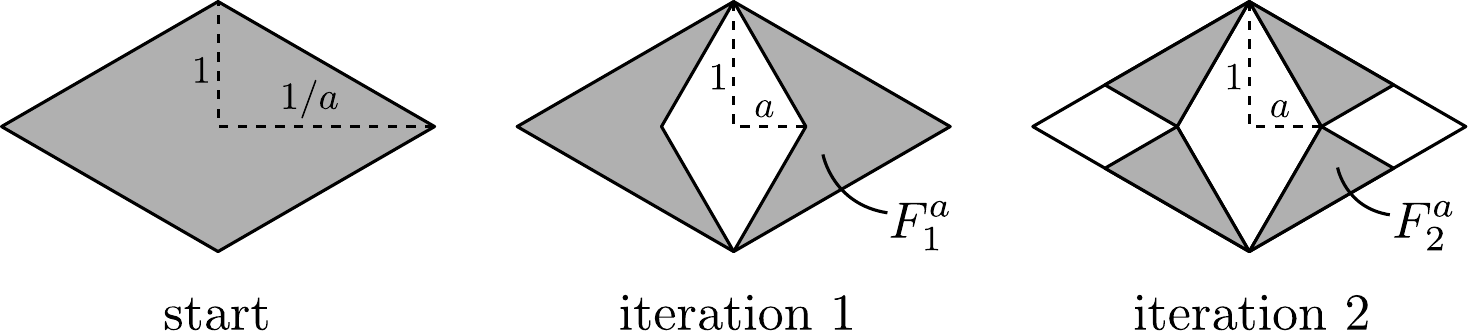}
    \caption{The first two iterations of the gasket construction. The shaded regions in iterations 1 and 2 are the empty polygons.}
    \label{fig:iter2}
\end{figure}
\begin{figure}[H]
    \centering
    \includegraphics[width=0.95\linewidth]{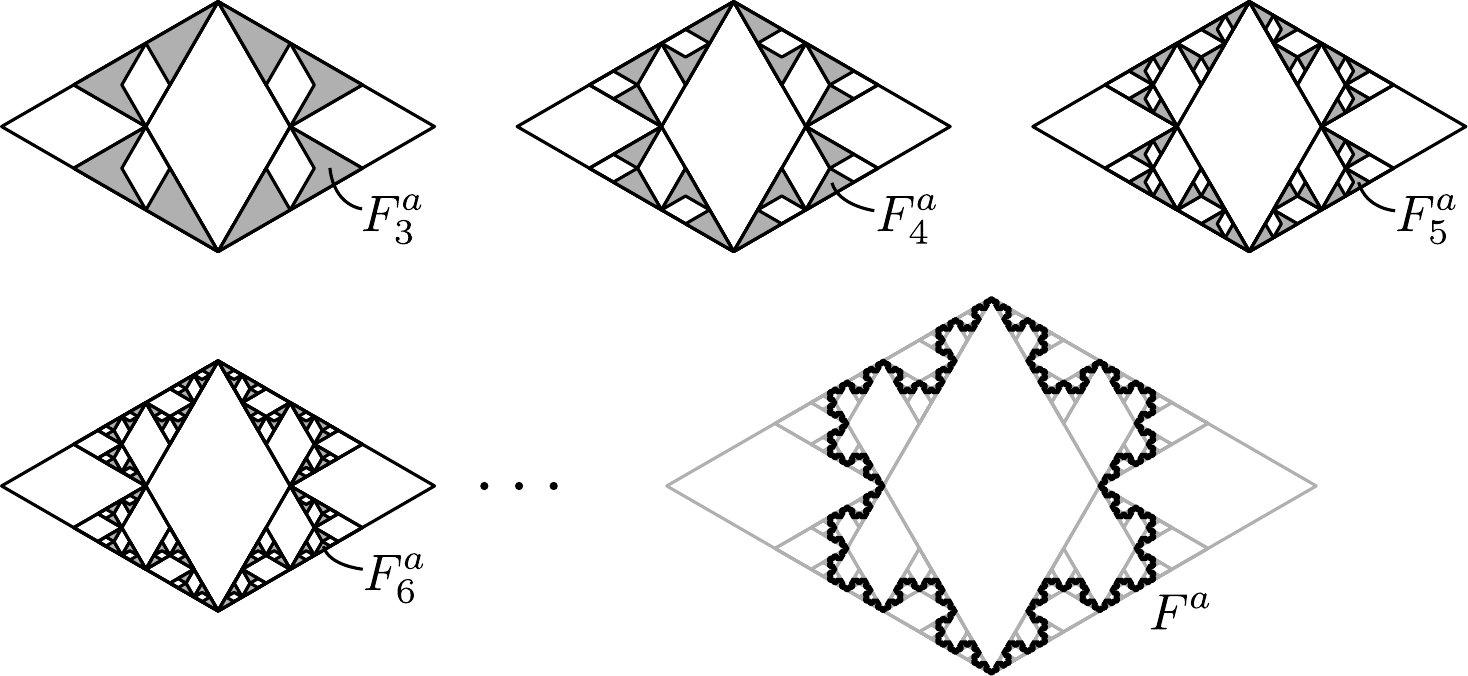}
    \caption{The next few iterations of the gasket construction and the final result.}
    \label{fig:gasket}
\end{figure}
For each $k$, let $F^a_k$ denote the union of all the (closed, filled) empty polygons at iteration $k$. We define the final fractal to be
\[F^a = \bigcap_{k=1}^\infty F^a_k\tp\]
Since $F^a$ is the intersection of nested nonempty compact sets, it is nonempty and compact. However, it is not yet clear that $F^a$ is a curve. To prove this, we construct a Jordan curve $f^a$, then prove that the image of $f^a$ equals $F^a$ (\cref{contfamily}).

It is easy to see that there are only two types (up to similarity) of empty polygons that arise, a right triangle and a dart. However, some of these similarities involve 90-degree rotations, interchanging the roles of the vertical and horizontal rhombi. It ends up being more convenient to focus only on the even-numbered iterations, which prevents this issue.

\begin{defn}\label{contact}
    A \emph{contact point} is a point where a vertical rhombus touches a horizontal rhombus. (For example, there are two contact points at iteration 1 and four contact points at iteration 2.)
\end{defn}

\begin{lem}\label{types}
    Within the \emph{even-numbered} iterations, the following hold:
    \begin{itemize}
        \item[(a)] The are only two types of empty polygons up to similarity. 
        \item[(b)] These similarities involve only translations, scaling, and reflections across the $x$ and $y$ axes. These similarities respect contact points and respect the rhombi that will be inserted in future iterations.
        \item[(c)] Each empty polygon contains exactly two contact points. In each iteration, the points where distinct empty polygons intersect are precisely the contact points.
    \end{itemize}   
\end{lem}
\begin{proof}
    Proceed by induction on the iteration $k$ ($k$ even). We call the two types (pictured in \cref{fig:types}) the \emph{wedge} and the \emph{dart}. The base case, iteration 2, consists of four wedges. For the induction step, suppose that the lemma holds at iteration $k$, examine each type, and see what happens at iteration $k+2$. Going from $k$ to $k+2$ involves placing vertical rhombi, then horizontal rhombi. One wedge at $k$ produces three wedges and a dart at $k+2$; one dart at $k$ produces two wedges and two darts at $k+2$.
    \begin{figure}[H]
        \centering
        \includegraphics[width=0.8\linewidth]{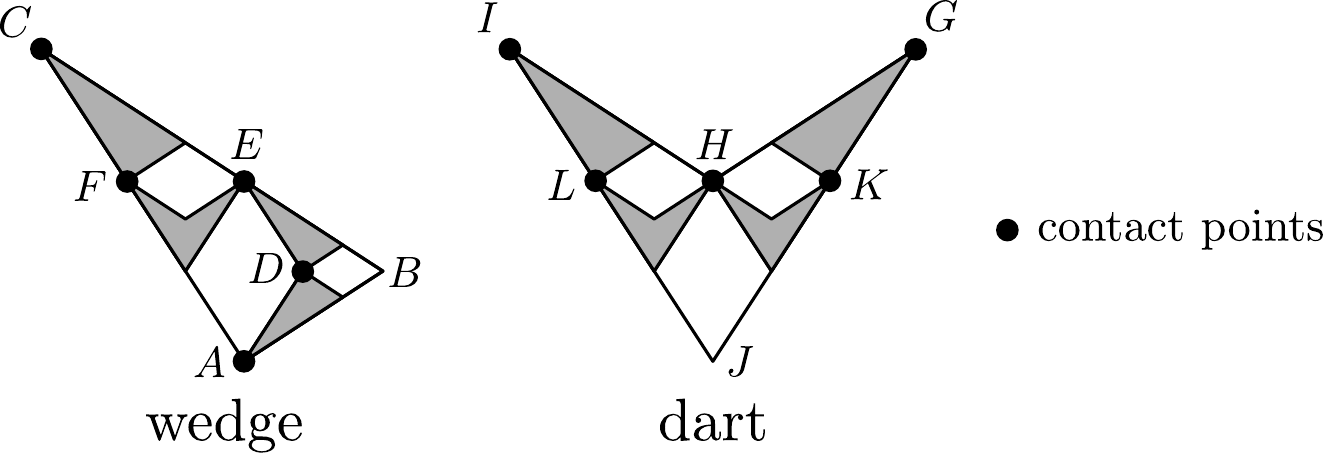}
        \caption{The two types of empty polygons. $ABC$ and $GHIJ$ are empty polygons at iteration $k$; the shaded polygons are empty polygons at $k+2$. The contact points at $k$ are $A,C,G,I$; the contact points at $k+2$ include those and $D,E,F,H,K,L$.}
        \label{fig:types}
    \end{figure}
    Points (b) and (c) for the induction step can be visually checked. That these polygons are truly similar follows from the fact that the slopes of all line segments involved are $\pm a$ or $\pm 1/a$.
\end{proof}

\begin{lem}\label{bounds}
    The following bounds hold:
    \begin{itemize}
        \item[(a)] Let $w_a$ (resp.\ $d_a$) represent the diameter of a wedge (resp.\ dart) of height $1$. Then $w_a,d_a < 2/a$. (The diameter of a nonempty set $E \sst \R^2$ is $\sup\{d(x,y) \mid x,y \in E\}$.)
        \item[(b)] If a wedge or dart has height $1$, then after two iterations, the heights of the resulting wedges and darts are less than $(1+a)/2$.
    \end{itemize}
\end{lem}
\begin{proof}
    Straightforward geometric calculations show that $w_a = \dfrac{(1+a^2)^{3/2}}{2a}$ and $d_a = \max\{2a, \sqrt{1+a^2}\}$. Since $0 < a < 1$, (a) follows. Point (b) follows from similar calculations that show that the heights of the resulting empty polygons are $a^2$ or $(1-a^2)/2$ (\cref{fig:bounds}), which are less than $(1+a)/2$ since $0 < a < 1$.
\end{proof}
\begin{figure}[H]
    \centering
    \includegraphics[width=0.75\linewidth]{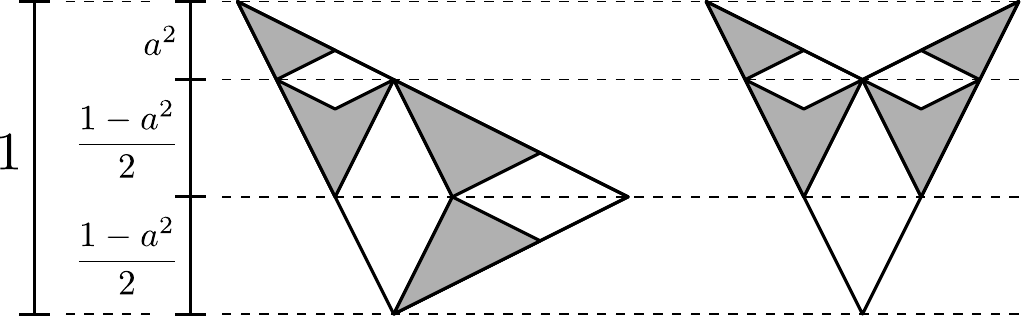}
    \caption{The relationships between the heights of empty polygons in consecutive even iterations.}
    \label{fig:bounds}
\end{figure}

This lemma is used in the proof of \cref{contfamily} to give a useful bound for the diameters of empty polygons.

\begin{defn}
    A \emph{Jordan curve} is a continuous function $\gamma : [0,1] \to \R^2$ such that $\gamma(0) = \gamma(1)$ and $\gamma$ is injective on $[0,1)$. (That is, a Jordan curve is a simple closed plane curve.)
\end{defn}

\begin{thm}\label{contfamily}
    There is a continuous function $f : (0,1) \times [0,1] \to \R^2$ such that for each $a \in (0,1)$, the map $t \mapsto f(a,t)$ is a Jordan curve whose image is $F^a$, the set created by the gasket construction. (In other words, the gasket construction yields a continuous family of Jordan curves.) Furthermore, for all $a$, the set $F^a$ is the closure of $C^a$, where $C^a$ denotes the set of all contact points arising in the construction of $F^a$.
\end{thm}
\begin{proof}
    We prove the theorem geometrically, in a similar fashion to \cite{geoproof}.
    
    We define a sequence of functions $f_1, f_2, \dots$ that converges to the desired $f$. It is helpful to think of each of $f, f_1, f_2, \dots$ as a family of curves indexed by $a$. To this end, we think of $a \in (0,1)$ as fixed, and let $f^a : [0,1] \to \R^2$ denote the curve $t \mapsto f(a,t)$. Similarly, we let $f_k^a$ denote the curve $t \mapsto f_k(a,t)$, and view ourselves to be defining $f^a$ as the limit of the curves $f_k^a$. (As before, the parameter $a$ is the aspect ratio of the rhombi involved in the gasket construction.)
    
    As before, we work using the even iterations, so the curve $f_k^a$ is based on the construction at iteration $2k$. The curves $f_1^a,f_2^a,\dots$ are defined inductively, with more detail added each time. First, $f_1^a$ maps the inputs $0,\, 1/4,\, 2/4,\, 3/4,\, 1$ to the four contact points at iteration 2 (since these are closed curves, $f_k^a(0) = f_k^a(1)$ always). This function maps each of the subintervals $[0,1/4],\,[1/4,2/4],\, \dots$ to a constant-speed path (see \cref{fig:curves12}) within one of the four empty polygons at this iteration. Next, $f_2^a$ maps $0,\, 1/16,\,2/16,\, \dots,\, 1$ to the 16 contact points at iteration 4, and maps each of the subintervals $[0,1/16],\,[1/16,2/16],\,\dots$ into one of the 16 empty polygons at this iteration. The curve $f_2^a$ equals $f_1^a$ at the inputs $0,\, 1/4,\, 2/4,\, 3/4,\, 1$, and $f_2^a$ maps each of $[0,1/4],\,[1/4,2/4],\, \dots$ into the same one of iteration 2's empty polygons that $f_1^a$ does.
    \begin{figure}[H]
        \centering
        \includegraphics[width=0.95\linewidth]{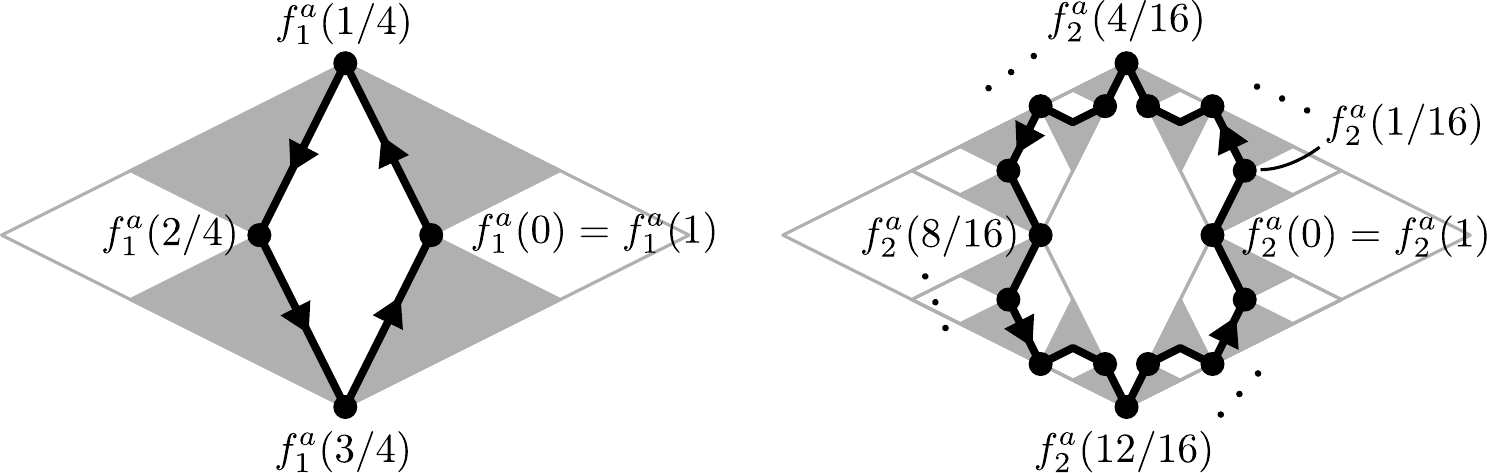}
        \caption{The curves $f_1^a$ and $f_2^a$.}
        \label{fig:curves12}
    \end{figure}
    This process continues inductively, with $f_k^a$ mapping $0,\,1/2^{2k},\,2/2^{2k},\,\dots,\,1$ to the $2^{2k}$ contact points of iteration $2k$. To define this inductive step explicitly, it suffices to detail what happens inside each empty polygon. The inductive step is defined so that, for any $k$ and $l$, the path that $f_k^a$ takes through a wedge of iteration $2k$ is similar to the path that $f_l^a$ takes through a wedge of iteration $2l$ (though the direction may flip). The corresponding statement for darts also holds. Thus, there are only two cases to describe, the wedge and the dart (\cref{fig:inductivestep}).
    \begin{figure}[H]
        \centering
        \includegraphics[width=0.85\linewidth]{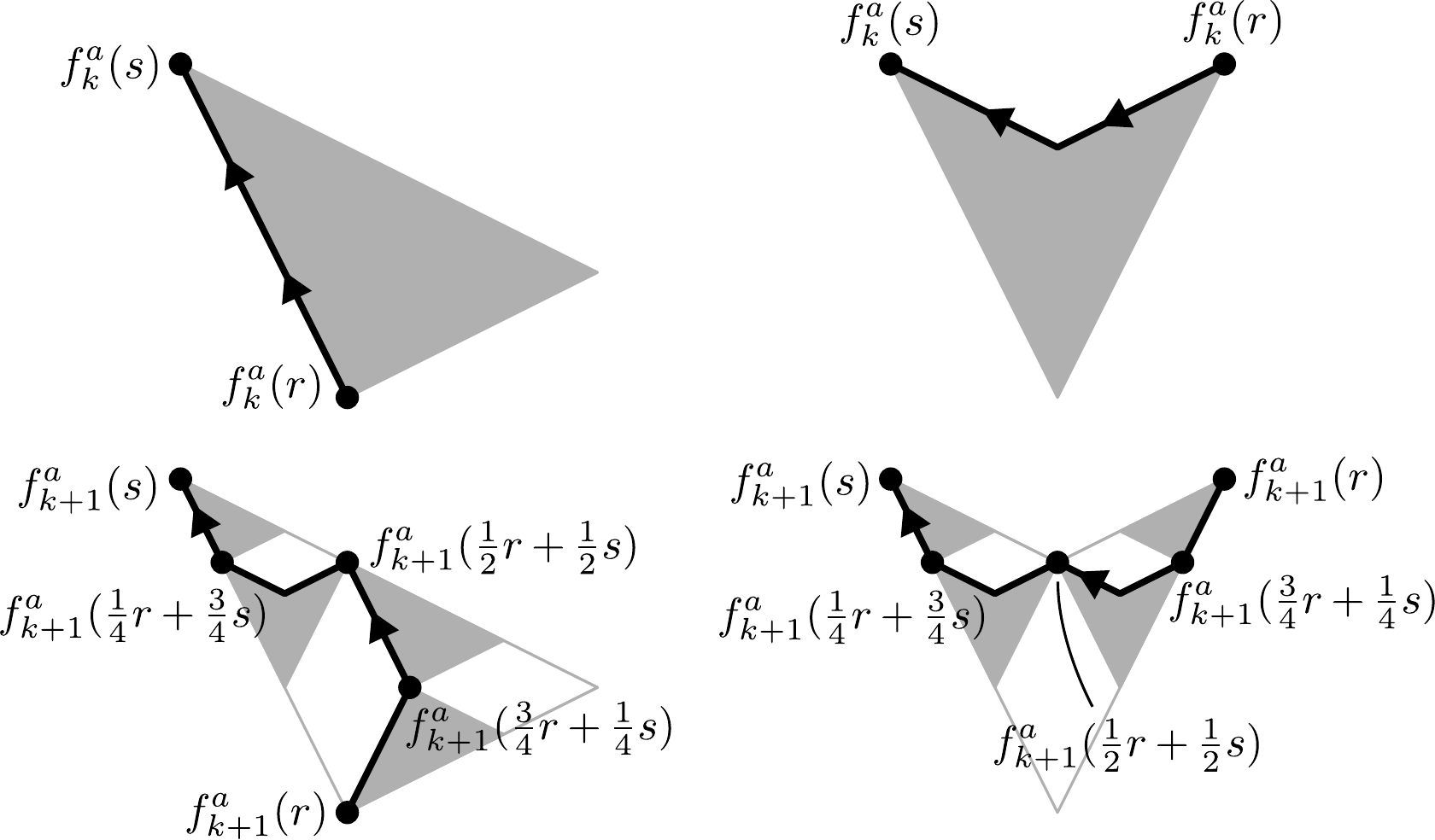}
        \caption{How $f_{k+1}^a$ is defined from $f_k^a$. Here $r,s$ are adjacent elements of $\{0, 1/2^{2k},\dots,1\}$, with either $r < s$ or $s < r$. The arrows in the figure represent the direction of the curves when $r < s$; the direction is reversed when $s < r$.}
        \label{fig:inductivestep}
    \end{figure}
    It is clear from the construction that each $f_k$ is continuous, i.e., that $f_k^a(t)$ is continuous in $(a,t)$. We now need to prove that the $f_k$ converge to a continuous function $f$. To do this, it suffices to prove that for every $\eps \in (0,1/2)$, the $f_k$ converge to a continuous $f$ on the domain $(\eps,1-\eps) \times [0,1]$; from this it follows that the $f_k$ converge to a continuous $f$ on all of $(0,1)\times[0,1]$. Fix $\eps \in (0,1/2)$, and define the constants $B_\eps = 2/\eps$, $r_\eps = 1-\eps/2$. By \cref{bounds}, for all $a \in (\eps,1-\eps)$, the diameters of wedges/darts of height $1$ are less than $2/a$, which is less than $2/\eps = B_\eps$. Furthermore, if $h$ is the maximum height of an empty polygon at iteration $2k$, then the heights of empty polygons at $2(k+1)$ are less than $\dfrac{1+a}{2}h$, which is less than $\dfrac{1+(1-\eps)}{2}h = r_\eps h$. These facts, and the fact that the four wedges at iteration 2 have height $1$, imply that the diameter of any empty polygon at iteration $2k$ is less than $B_\eps\, r_\eps^{k-1}$. This quantity tends to 0 as $k$ grows.
    
    Now, consider any $k,l$ with $k \leq l$ and any $t \in [0,1]$. The point $f_k^a(t)$ is in one of the empty polygons at iteration $2k$, and by construction, $f_l^a(t)$ is in this same empty polygon. Then $d\big(f_k^a(t),\, f_l^a(t)\big)$ is at most the diameter of this polygon, which is less than $B_\eps\,r_\eps^{k-1}$. Thus, the sequence $f_1,f_2,\dots$, where the functions are restricted to the domain $(\eps,1-\eps) \times [0,1]$, is uniformly Cauchy. This means the $f_k$ converge uniformly to a function $f$ on this domain. By the uniform limit theorem, $f$ is continuous on $(\eps,1-\eps) \times [0,1]$. Since $\eps \in (0,1/2)$ was arbitrary, this means the $f_k$ converge to a continuous $f$ on all of $(0,1) \times [0,1]$.
    
    We now prove that $f^a$ is a Jordan curve for all $a$. We have $f^a(0) = f^a(1)$ by construction, so it remains to show that $f^a$ is injective on $[0,1)$. Consider any $r,s \in [0,1)$ with $r < s$. Pick $k$ large enough that $s-r > 2/2^{2k}$ and $s < 1 - 1/2^{2k}$. Then there is some $i \in \{0,1,\dots,2^{2k}\}$ such that $r < i/2^{2k} < (i+1)/2^{2k} < s$. We know that $f_k^a$ maps each of the subintervals $[0,1/2^{2k}],\,[1/2^{2k},2/2^{2k}],\,\dots$ into a distinct empty polygon of iteration $2k$, and two distinct empty polygons intersect only if they correspond to adjacent subintervals (with the exception of the pair $[1-1/2^{2k},1]$, $[0,1/2^{2k}]$\,). Since $r$ and $s$ are in non-adjacent subintervals and $s \notin [1-1/2^{2k},1]$, $f_k^a$ maps $r$ and $s$ into disjoint empty polygons. Since $f_l^a$ maps $r$ and $s$ into these same empty polygons for all $l \geq k$, and these polygons are closed, $f^a$ also maps $r$ and $s$ into these disjoint polygons. Thus, $f^a(r) \neq f^a(s)$.
    
    Finally, we show that the image of $f^a$ is $F^a$, and that this equals the closure of $C^a$. If $q \in [0,1]$ is a dyadic rational, then as $k$ grows, $f_k^a(q)$ is eventually a contact point, and subsequently does not change. This means $f^a(q)$ is a contact point. Thus, $f^a$ maps dyadic rationals to contact points. For any $t \in [0,1]$, there is a sequence of dyadic rationals in $[0,1]$ that converges to $t$.  Thus, since $f^a$ maps dyadic rationals to contact points and $f^a$ is continuous, there is a sequence of contact points that converges to $f^a(t)$. Now consider any $p \in F^a$. By definition, $p \in F_k^a$ for all $k$, meaning that $p$ is in one of the empty polygons at iteration $k$; these empty polygons can be made as small as desired by picking large $k$. Since each empty polygon contains a contact point, this means $p$ is arbitrarily close to some contact point. Each point of $f^a([0,1])$ is arbitrarily close to some element of $C^a$, and each point of $F^a$ is arbitrarily close to some element of $C^a$. Since $f^a([0,1])$ and $F^a$ are both closed sets that include $C^a$, this means they must both equal the closure of $C^a$.
\end{proof}

\section{Relationship to the standard Koch curve construction}\label{sec:koch}
We can now prove that the gasket construction actually does give the Koch snowflake when the aspect ratio $a$ equals $1/\sqrt{3}$. We give a brief definition of the Koch snowflake that is suited to this paper's arguments.

\begin{defn}\label{defkoch}
    The \emph{Koch snowflake} curve $g : [0,1] \to \R^2$ is defined to be the limit of the curves $g_1,g_2,\dots$ that are defined inductively as in \cref{fig:koch}. Each $g_k$ is a constant-speed curve tracing a polygon centered at the origin, mapping the inputs $0,\,1/(3\cdot4^k),\,2/(3\cdot4^k),\,\dots,\,1$ to the polygon's $3\cdot4^k$ vertices. Constructing $g_{k+1}$ from $g_k$ involves replacing the middle third of each line segment of $g_k$ by two segments of the middle third's length, forming a spike that points toward the exterior of the curve. (Similarly to before, the sequence $g_1,g_2,\dots$ can be shown to be uniformly Cauchy, so it converges to a continuous $g$.)
    \begin{figure}[H]
        \centering
        \includegraphics[width=0.9\linewidth]{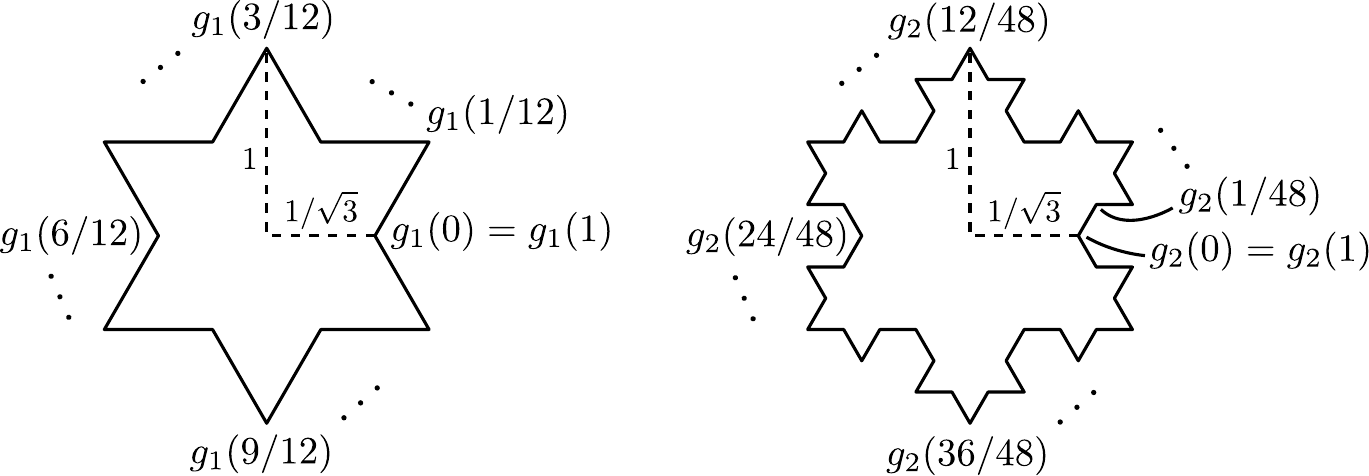}
        \caption{The curves $g_1$ and $g_2$.}
        \label{fig:koch}
    \end{figure}
\end{defn}

\begin{thm}\label{equalskoch}
    The gasket construction yields the Koch snowflake when $a = 1/\sqrt{3}$. That is, if $a = 1/\sqrt{3}$, then the set $F^a$ equals the image of $g$.
\end{thm}
\begin{proof}
    Suppose $a = 1/\sqrt{3}$. We prove that the image of $g$ is the closure of $C^a$. We claim the following:
    \begin{itemize}
        \item[(a)] For any $k$, the image of $g_k$ falls within the empty polygons of iteration $2k$.
        \item[(b)] For any $k,l$, the path $g_k$ takes through a wedge of iteration $2k$ is similar to the path $g_l$ takes through a wedge of iteration $2l$. The corresponding statement for darts also holds. Though these similarities may flip the direction of travel, they preserve the direction toward the exterior of the curve (i.e., the direction spikes point).
        \item[(c)] For any $k$, $g_k$ maps the inputs $0, 1/2^{2k},\dots,1$ onto the $2^{2k}$ contact points of iteration $2k$ (with $g_k(0) = g_k(1)$\,).
    \end{itemize}
    These claims can by checked by induction on $k$; it suffices to check the base case, then check the induction step for the wedge and dart (\cref{fig:kochproof}). 
    
    If $q \in [0,1]$ is a dyadic rational, then as $k$ grows, $g_k(q)$ is eventually a contact point, and subsequently does not change. This means $g(q)$ is a contact point. Thus, $g$ maps dyadic rationals to contact points. For any $t \in [0,1]$, there is a sequence of dyadic rationals in $[0,1]$ that converges to $t$.  Thus, since $g$ maps dyadic rationals to contact points and $g$ is continuous, there is a sequence of contact points that converges to $g(t)$. Each point of $g([0,1])$ is arbitrarily close to some point of $C^a$, and $g([0,1])$ is a closed set that includes $C^a$, so $g([0,1]) = \overline{C^a} = F^a$.
\end{proof}
\begin{figure}[H]
    \centering
    \includegraphics[width=0.85\linewidth]{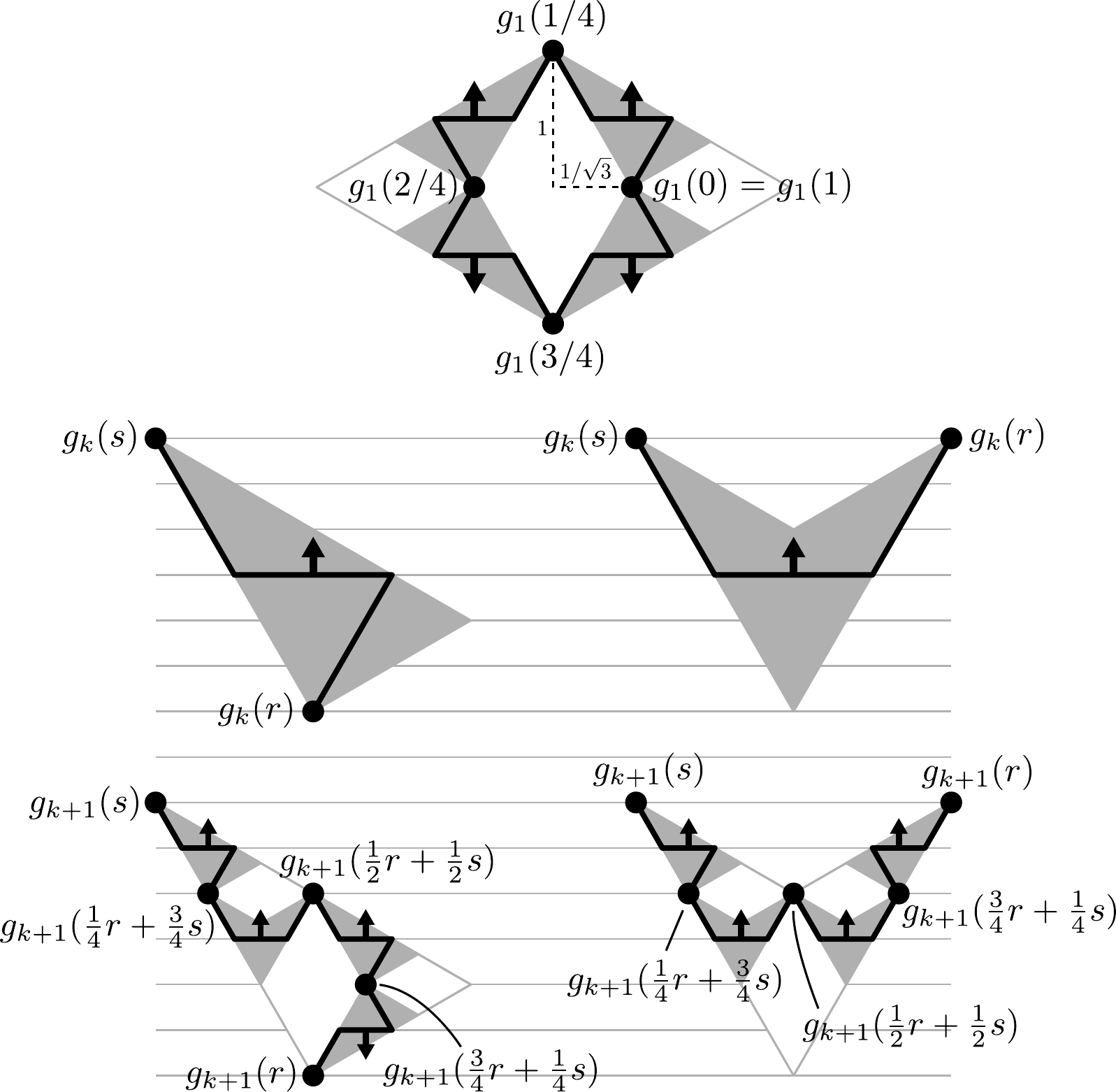}
    \caption{The base case and induction step for \cref{equalskoch}. Here $r,s$ are adjacent elements of $\{0, 1/2^{2k},\dots,1\}$, with either $r < s$ or $s < r$. The arrows represent the direction toward the exterior of the curve. The grid lines assist in checking that the claimed similarities are actually similarities.}
    \label{fig:kochproof}
\end{figure}

Using similar techniques, it can be shown that when $a = 1/\sqrt{3}$, the function $f^a$ (as defined in the proof of \cref{contfamily}) actually equals $g$.

\section{Hausdorff dimension and area}\label{sec:hausdorffdim}
Definitions \ref{similarity}, \ref{osc} and \cref{hausdorff} summarize the relevant part of Falconer (2014), pp.\ 133--140 \cite{falconer}.

\begin{defn}\label{similarity}
    A transformation $S : \R^n \to \R^n$ is a \emph{similarity} of \emph{ratio} $c > 0$ if $d(S(x),S(y)) = cd(x,y)$ for all $x,y \in \R^n$. The map $S$ is a \emph{contracting similarity} if $c < 1$.
\end{defn}

\begin{defn}\label{osc}
    A family of contracting similarities $S_1,\dots,S_m$ satisfies the \emph{open set condition} if there is a nonempty, bounded, open set $V$ such that $\bigcup_{i=1}^m S_i(V) \sst V$, with this union disjoint.
\end{defn}

\begin{thm}\label{hausdorff}
    Suppose $S_1,\dots,S_m$ is a family of contracting similarities, with ratios $c_1,\dots,c_m$ respectively. Then there is a unique nonempty compact set $F$ such that $F = \bigcup_{i=1}^m S_i(F)$. Furthermore, if $S_1,\dots,S_m$ satisfy the open set condition, then the Hausdorff dimension of $F$ is the unique $s \geq 0$ such that $c_1^s + \dots + c_m^s = 1$. (This $s$ is unique because the map $s \mapsto c_1^s + \dots + c_m^s$ is strictly decreasing. The purpose of the open set condition is to prevent the similarities from ``overlapping too much''.)
\end{thm}

We compute the Hausdorff dimension and interior area of the curve using the self-similarity relationship in \cref{fig:similarities}. (Note that we are no longer just considering the even-numbered iterations.) Let $E^a$ represent the upper-right quadrant of $F^a$ (including the endpoints). The figure gives similarities $S_1,S_2,S_3$ with ratios $a,\,(1-a^2)/2,\,(1-a^2)/2$ respectively, where the nonempty compact set $E^a$ satisfies $E^a = \bigcup_{i=1}^3 S_i(E^a)$. These similariti\textbf{e}s also satisfy the open set condition by choosing $V$ to be the (open, filled) triangle $ABC$. Thus, the Hausdorff dimension of $E^a$ (which is the Hausdorff dimension of the whole curve) is the unique $s$ such that
\begin{equation}\label{gasketdim}
    a^s + \left(\frac{1-a^2}{2}\right)^s + \left(\frac{1-a^2}{2}\right)^s = a^s + 2\left(\frac{1-a^2}{2}\right)^s = 1\tp
\end{equation}
\begin{figure}[H]
    \centering
    \includegraphics[width=0.55\linewidth]{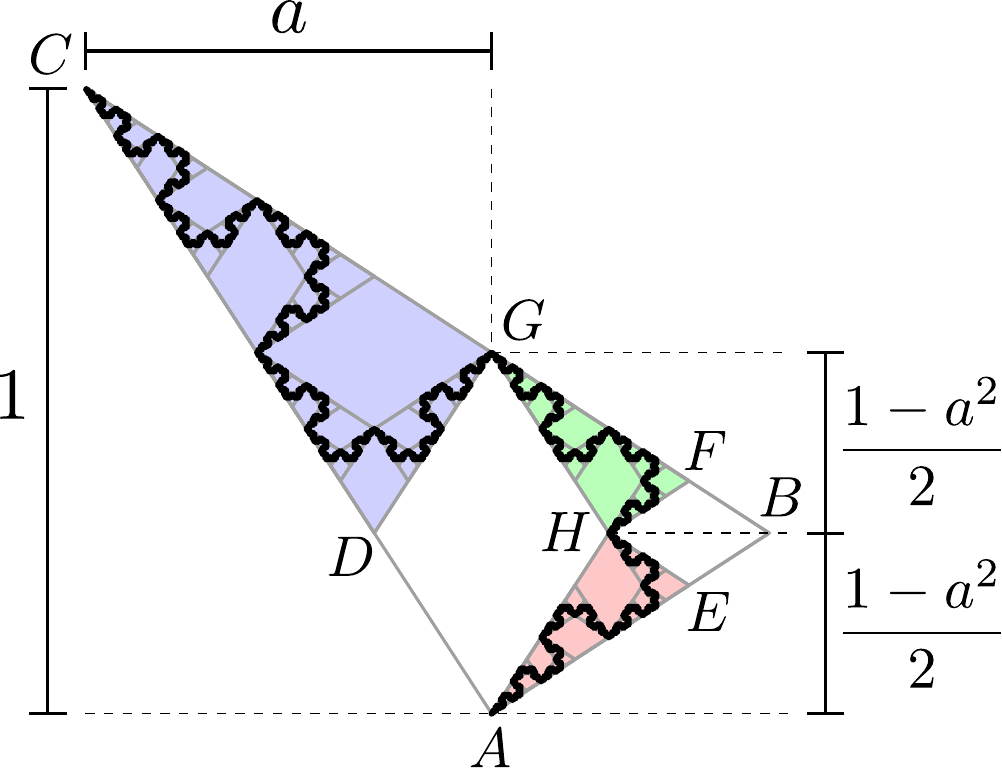}
    \caption{The set $E^a$ partitioned into three parts similar to the whole. The map $S_1$ takes $A,B,C$ to $G,D,C$ respectively, $S_2$ takes $A,B,C$ to $H,F,G$ respectively, and $S_3$ takes $A,B,C$ to $H,E,A$ respectively.}
    \label{fig:similarities}
\end{figure}

\begin{rmk}[Background on Hausdorff dimension]\label{rmk:hausdorff}
    The Hausdorff dimension of an arbitrary set $F \sst \R^n$ is defined using the $s$-dimensional Hausdorff \emph{measure} of $F$, denoted $\mathcal{H}^s(F)$ ($s \geq 0$). Hausdorff measure has the following properties:
    \begin{itemize}
        \item[(a)] If $S$ is a similarity of ratio $c$, then $\mathcal{H}^s(S(F)) = c^s\,\mathcal{H}^s(F)$ for all $s$.
        \item[(b)] There is a unique $s$ where $\mathcal{H}^s(F)$ jumps from $\infty$ to 0. That is, there is $s$ such that for all $t < s$, $\mathcal{H}^t(F) = \infty$, and for all $t > s$, $\mathcal{H}^t(F) = 0$. This $s$ is called the \emph{Hausdorff dimension} of $F$, denoted $\dimH F$. At this $s$, $\mathcal{H}^s(F)$ can take any value in $[0,\infty]$.
    \end{itemize}
    If it were always the case that, if $s = \dimH F$, then $0 < \mathcal{H}^s(F) < \infty$, then \cref{hausdorff} would follow easily for self-similar fractals. For example, the similarities in \cref{fig:similarities} show that
    \[\mathcal{H}^s(E^a) = \mathcal{H}^s(S_1(E^a)) + \mathcal{H}^s(S_2(E^a)) + \mathcal{H}^s(S_3(E^a))\]
    \[= a^s\, \mathcal{H}^s(E^a) + 2\left(\frac{1-a^2}{2}\right)^s \mathcal{H}^s(E^a)\tc\]
    which gives \cref{gasketdim} after dividing the equation by $\mathcal{H}^s(E^a)$. However, it is not true that for all sets $F$, there exists an $s$ such that $0 < \mathcal{H}^s(F) < \infty$. The difficult part of proving \cref{hausdorff} is showing that $0 < \mathcal{H}^s(F) < \infty$ when the hypotheses of the theorem hold. For a definition of Hausdorff measure and a proof of this theorem, see \cite{vde} pp.\ 11, 23--32.
\end{rmk}

\vspace{6pt}
Equation \cref{gasketdim} can be used to numerically plot $\dimH F^a$ as a function of $a$ (\cref{fig:plot}). The Hausdorff dimension of the curve takes its maximum value at $a = 1/\sqrt{3}$, when the gasket construction gives the original Koch snowflake.
\begin{figure}[H]
    \centering
    \includegraphics[width=0.75\linewidth]{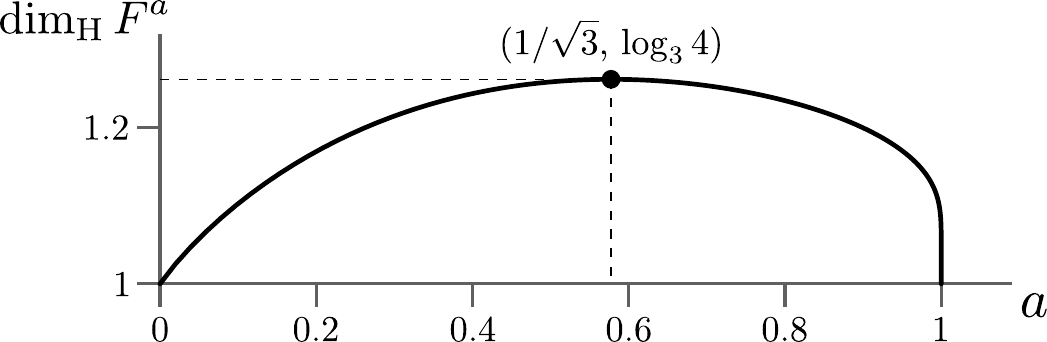}
    \caption{The plot of $\dimH F^a$ as a function of $a$.}
    \label{fig:plot}
\end{figure}

\begin{prop}\label{gasketdimmax}
    The Hausdorff dimension $\dimH F^a$ is maximized when $a = 1/\sqrt{3}$, attaining the value $\log_3 4$. Furthermore, $\lim\limits_{a\searrow0} \dimH F^a = \lim\limits_{a\nearrow1} \dimH F^a = 1$.
\end{prop}
\begin{proof}
    The values $a = 1/\sqrt{3}$, $s = \log_3 4$ satisfy the equation $a^s + 2\left(\frac{1-a^2}{2}\right)^s = 1$. By applying implicit differentiation to this equation, it can be seen that $\dfrac{\mathrm{d}s}{\mathrm{d}a} = 0$ and $\dfrac{\mathrm{d}^2s}{\mathrm{d}a^2} < 0$ at $a = 1/\sqrt{3}$, $s = \log_3 4$, which shows that this point is at least a local maximum. We omit the rest of the proof.
\end{proof}

Though $\lim\limits_{a\nearrow1} \dimH F^a = 1$, this convergence is very slow. For example, when $a = 1-10^{-16}$, $\dimH F^a$ is approximately $1.018$. This means that the curve retains significant fractal ``roughness'' even when it is, distance-wise, extremely close to the degenerate curve with $a = 1$ (which is just a square).

\begin{prop}\label{curvearea}
    The region enclosed by the curve $F^a$ has area $\dfrac{8a}{1+4a^2-a^4}$.
\end{prop}
\begin{proof}
    Since $\dimH F^a < 2$, the 2-dimensional Hausdorff measure of $F^a$ is zero. 2D Hausdorff measure is proportional to 2D Lebesgue measure (\cite{falconer} pp.\ 45--46), so the curve itself has zero area. Let $x$ be the area of the shaded part of $ABC$ (\cref{fig:area}). Partition this region into four subregions: the rhombus $AHGD$ and the shaded parts of $GDC$, $HFG$, and $HEA$. The similarities $S_1,S_2,S_3$ from before give us the areas of these subregions in terms of $x$. Note that $S_1$ (mapping $ABC$ to $GDC$) maps the shaded part of $ABC$ to the \emph{unshaded} part of $GDC$. Writing $x$ as the sum of the areas of these subregions gives the equation
    \[x = \operatorname{area}(AHGD) + \big[\!\operatorname{area}(GDC) - a^2x\big] + 2\left(\frac{1-a^2}{2}\right)^2x\]
    \[x = \frac{a(1-a^2)^2}{2} + \left[\frac{a(1-a^4)}{4} - a^2x\right] + 2\left(\frac{1-a^2}{2}\right)^2x\tp\]
    Solving this equation for $x$ gives $x = \dfrac{a(3-4a^2+a^4)}{2(1 + 4a^2 - a^4)}$. Then the area of the whole region enclosed by the curve is
    \[\operatorname{area}(ACIJ) + 4x = 2a + 4\frac{a(3-4a^2+a^4)}{2(1 + 4a^2 - a^4)} = \frac{8a}{1+4a^2-a^4}\tp\]
\end{proof}
\begin{figure}[H]
    \centering
    \includegraphics[width=0.9\linewidth]{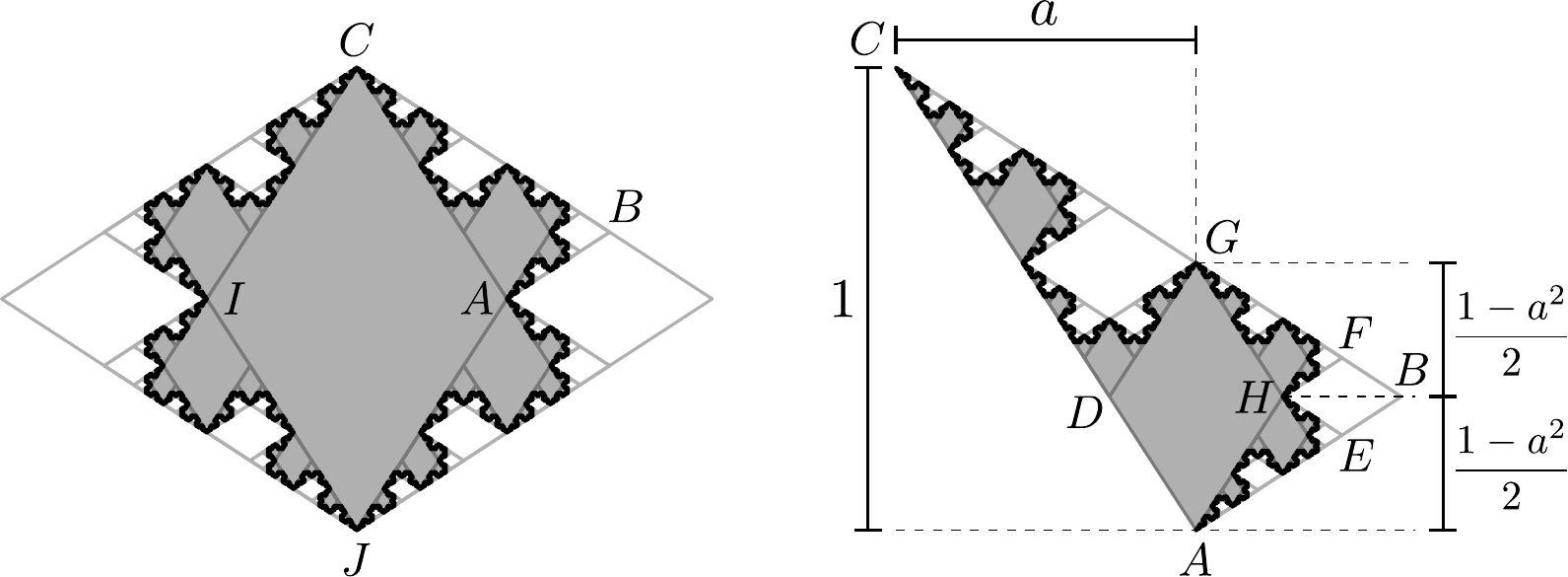}
    \caption{The region enclosed by the curve $F^a$.}
    \label{fig:area}
\end{figure}

The area attains its maximum at $a = 1/\sqrt{3}$, though this is an artifact of the arbitrary decision to keep the height of the construction constant (as opposed to the width or some other quantity).

\end{document}